\documentclass[a4paper]{article}

\usepackage{
amsmath,
amsthm,
amscd,
amssymb,
}
\usepackage{comment}
\usepackage{tikz}
\usetikzlibrary{positioning,arrows,calc}
\usepackage{xspace}

\usepackage{graphicx}

\usepackage{url}

\theoremstyle{plain}

\newtheorem{theorem}{Theorem}

\newtheoremstyle{derp}
{3pt}
{3pt}
{}
{}
{\upshape}
{:}
{.5em}
{}
\theoremstyle{derp}

\newcommand{\Z}{\mathbb{Z}}

\newcommand{\N}{\mathbb{N}}

\title{Nonexpansive chaotic almost minimal systems on residually finite groups}

\author{
Ville Salo \\
vosalo@utu.fi
}

\begin{document}
\maketitle

\begin{abstract}
The question of existence of nonexpansive chaotic almost minimal (CAM) systems, and the existence of CAM systems on every residually finite group, were raised in a recent paper of Van Cyr, Bryna Kra and Scott Schmieding. We construct nonexpansive CAM systems on all finitely-generated residually finite groups.
\end{abstract}

\section{Introduction}

Recently, the notion of a \emph{CAM} (for chaotic almost minimal) system was introduced in \cite{CyKrSc24} by Van Cyr, Bryna Kra and Scott Schmieding. These are precisely the faithful and topologically transitive compact metrizable systems which have dense periodic points and no proper infinite subsystems.

What makes these systems interesting is that the last two properties are not that commonly seen in conjunction. However, the Furstenberg $(\times 2, \times 3)$-system has these properties. Thus, it is interesting to see whether these properties already have some implications on invariant measures, as understanding these for the $(\times 2, \times 3)$-system is a major open problem.

The paper \cite{CyKrSc24} raises several technical questions. In particular, they ask whether there exists a nonexpansive CAM system, and they ask whether every residually finite group admits a CAM system. In the present paper, we prove the following theorem, which solves roughly the first question, and solves the second for finitely-generated groups.

\begin{theorem}
Let $G$ be any finitely-generated residually finite group. Then $G$ admits a nonexpansive CAM action on a compact metrizable space.
\end{theorem}

Our systems are abstract quotients of subshifts.

We have not given much thought to whether such systems exist for non-finitely-generated groups (they ask the question in this generality). Our proof does not go through verbatim, as it uses the fact that finite systems are of finite type, which is true only on finitely-generated groups.

\section{Definitions}

By a \emph{$\Z$-system} we mean a compact metric space $X$ together with a continuous $\Z$-action, equivalently a fixed homeomorphism $\sigma : X \to X$ (which corresponds to the action of $1 \in \Z$. We write it as $(X, \sigma)$.

A \emph{periodic point} in $X$ is $x \in X$ such that $\sigma^n(x) = x$ for some $n \geq 1$, and $n$ is a \emph{period} of $x$. A $\Z$-system is \emph{faithful} if for all $n > 0$, there exists $x \in X$ such that $\sigma^n(x) \neq x$ (equivalently, no $n$ is a period for every point in $X$). It is \emph{topologically transitive} if for any nonempty open sets $U, V \subset X$, there exists $n \in \Z$ such that $\sigma^n(U) \cap V \neq \emptyset$. We say it is \emph{positively topologically transitive} if we can always take $n \geq 0$. It is \emph{expansive} if there exists $\epsilon > 0$ such that $\forall x, y: x \neq y \implies \exists n: d(\sigma^n(x), \sigma^n(y)) > \epsilon$.

A \emph{factor} of a $\Z$-system $(X, \sigma)$ is a $\Z$-system $(Y, \sigma')$ such that there is a \emph{factor map} $\phi : X \to Y$ meaning a surjective continuous map such that $\phi \circ \sigma = \sigma' \circ \phi$. Of course if such $\phi$ exists, the kernel pair $\{(x, x') \in X^2 \;|\; \phi(x) = \phi(x')\}$ is a closed $\sigma$-invariant equivalence relation. Conversely, every such relation $E \subset X^2$ gives rise to a factor $\phi : X \to X/E$. 

The notion of a \emph{CAM} (for chaotic almost minimal) system was introducd in \cite{CyKrSc24}. These systems are defined by the following properties:
\begin{itemize}
\item the action of $X$ is faithful,
\item the action of $\Z$ on $X$ is topologically transitive,
\item the periodic points are dense in $X$,
\item there is no proper infinite subsystem.
\end{itemize}
It is easy to see that when periodic points are dense, topological transitivity implies positive topological transitivity.

The \emph{full $\Z$-shift} is $(A^\Z, \sigma)$, where $A^\Z$ has the product topology, and $A$ is a finite set with the discrete topology (it is a Cantor set), and $\sigma : A^\Z \to A^\Z$ is the left shift, defined by $\sigma(x)_i = x_{i + 1}$. A \emph{subshift} is a closed shift-invariant subset of $A^\Z$ under the restricted dynamics of $\sigma$. All subshifts are expansive.

These notions can be generalized to actions of other groups $G$. In a $G$-system, $G$ acts on a compact metrizable space by homeomorphism. On the $G$-full shift $A^G$, $G$ acts by the formula $gx_h = x_{g^{-1}h}$. By a periodic point we mean one whose stabilizer is of finite index. Topological transitivity means that for any nonempty open $U, V$, we have $gU \cap V \neq \emptyset$ for some $g \in G$. This is equivalent to \emph{point transitivity} meaning the existence of $x \in X$ such that $Gx$ is dense. Expansivity means $\inf_{x \neq y} \sup_{g \in G} d(gx, gy) > 0$.

Our intervals are discrete, i.e.\ $[i, j] = \{k \in \Z \;|\; i \leq k \leq j\}$. A \emph{word} is an element of the free monoid $A^*$ generated by $A$. We write the product (which corresponds to concatenation) as just juxtaposition $u, v \mapsto uv$. It is also often useful to think of words as being \emph{positioned}, in which case they are functions $u : I \to A$ where $I \subset \Z$ is a finite interval. We can turn any positioned word into a word by moving $I$ to a canonical position such as $[0, |I|-1]$.

If $x \in A^\Z$, a \emph{subword} is a positioned word $u : I \to A$ such that $x_i = u_i$ for all $i \in I$, and we similarly define subwords of positioned words. If $u$ is a word, then $u^\Z$ is the $\sigma$-periodic point $x \in A^\Z$ defined by $x_{m|u| + i} = u_i$ for all $m \in \Z$ and $i \in [0, |u|-1]$. The \emph{length} $|u|$ of a word is the obvious one, defined inductively by $|\varepsilon| = 0$ and $|ua| = |u| + 1$ for $a \in A$, where $\varepsilon$ is the \emph{empty word} (the unique word of length $0$).

The \emph{language} of a $\Z$-subshift $X \subset A^\Z$ is $L(X) = \{u \in A^* \;|\; \exists x \in X: x_{[0,|u|-1]} = u\} \subset A^*$. The language of a $\Z$-subshift is always closed under passing to subwords, and is extendable meaning $u \in L(X) \implies \exists a, b \in A: aub \in L(X)$. Conversely, every factor-closed extendable language is the language of a unique $\Z$-subshift.


In the case of a $G$-subshift $X \subset A^G$, if $D \Subset G$ (the notation means finite subset) then $P \in A^D$ is called a \emph{pattern}. We write $x|D$ for the pattern $P \in A^D$ defined by $P_h = x_h$.

\section{A non-expansive CAM system on $\Z$}

Scott Schmieding asked whether there exists a group that admits a non-expansive CAM system. In this section we show that the integers do. In the following section, we generalize this construction for all residually finite groups.

\begin{theorem}
There is a CAM subshift on the integers, which admits a non-expansive CAM quotient.
\end{theorem}

\begin{proof}
We fix the alphabet to be $A = \{0,1\}$, and for each $n$ we will construct two words $w_{2n}, w_{2n+1} \in \{0,1\}^*$ which only differ in one letter. We start with $w_a = a$ for $a \in \{0,1\}$.

Suppose now $w_k$ has been defined for $k < 2n$. For $a \in \{0, 1\}$ we define
\[ w_{2n+a} = w_{2n-2} w_a w_{2n-2} \prod_{k = 0}^{2n-1} (w_k^n w_{2n-2}). \]
The product is taken with respect to concatenation, and ordered from left to right ($w_0^n w_{2n-2}$ being the leftmost word).

Let $X$ be the smallest subshift containing the periodic points $w_n^\Z$. We claim it is CAM.

Faithfulness: Since $w_{2n}$ and $w_{2n+1}$ clearly differ in only position $|w_{2n-2}|$, the points $x = \lim_n \sigma^{|w_{2n-2}|}(w_{2n}^\Z)$ and $y = \lim_n \sigma^{|w_{2n-2}|}(w_{2n+1}^\Z)$ differ only at the origin. From this, it is clear that the shift action is faithful on $X$. (These limits exist, but in case this is not clear to the reader, any limit points will do.)

Topological transitivity: The set of words contained in some $w_n^\Z$ is factor-closed and extendable, so the language of $X$ is indeed precisely that set of words (note that this does \emph{not} mean that $X$ does not contain other points). In the construction of $w_{2n+a}$ we include high powers of all previous words, so subwords of points $w_n^\Z$ are coincide with the words that appear in the $w_n$ themselves, and indeed with the words that appear in $w_n$ for all large enough $n$. Topological transitivity is now clear. Finally, the language also coincides with the words that appear in $w_{2n}$ for large enough $n$.

Density of periodic points: This is true by definition.

Nonexistence of infinite proper subsystems: Suppose $Y \subset X$ is infinite. We claim that then the language of $Y$ contains every word $w_{2k}$. Consider $w_{2k}$ for fixed $k$. For any $n \in \N$, we write the word $w_{2n}$ as a concatenation of positioned words $w_i$ with $i \leq 2k$ in a specific way: rewrite positioned subwords $w_m$ for $m > 2k$ inductively by using the explicit formula in the definition. Repeating this, we obtain a decomposition $w_{2n} = u_1 u_2 \cdots u_h$ where $u_i \in \{w_0, w_1, \ldots, w_{2k}\}$ for all $i \in \{1, \ldots, h\}$.

For example, let $n = 2$. If $k \geq 2$, then the decomposition of $w_4$ would be simply $w_4$. If $k = 1$, then we use the definition to rewrite $w_4 = w_{2n} = w_2 w_0 w_2 w_0^2 w_2 w_1^2 w_2 w_2^2 w_2 w_3^2 w_2$, and then further rewrite the subword $w_3 = w_{2 \cdot 1 + 1} = w_0 w_1 w_0 w_0 w_0 w_1 w_0$ to get the final decomposition
\[w_4 = w_2 w_0 w_2 w_0 w_0 w_2 w_1 w_1 w_2 w_2 w_2 w_2 w_0 w_1 w_0 w_0 w_0 w_1 w_0 w_0 w_1 w_0 w_0 w_0 w_1 w_0 w_2. \]

Let $J = 2|w_{2k+2}|$. We claim that for the decomposition of $w_{2n}$ described above, we have
\begin{itemize}
\item $u_1 = u_h = w_{2k}$
\item whenever $u_t u_{t+1} = w_i w_j$ with $i \neq j$, we have $u_s = w_{2k}$ for some $s \in \{t-J, \ldots, t+J\}$. 
\end{itemize}
This is rather trivially true for the decomposition of $w_4$ with $k = 1$, which was given above, since $J = 2|w_4|$

We prove this by induction on $n$. Both are trivially true for $n \leq k$, as then $w_{2k}$ is its own decomposition and $h = 1$. The first item to be proved clearly stays true inductively. For the second item, from the general form of $w_{2n}$, we see that if $w_i w_j$ occurs as $u_t u_{t+1}$, then
there are several cases to consider. First, it may happen that $u_t u_{t+1}$ occurs properly inside the recursive decomposition of some $w_j$-subword. In this case, the claim follows by induction as $w_{2k}$ will appear within $J$ steps even in the recursive decomposition of this $w_j$ (noting that we use the same rule to decompose $w_j$ on its own, and as a subword of $w_{2n}$).

It may also happen that $u_t u_{t+1}$ appears on the boundary between two words in the first decomposition of $w_{2n}$. This cannot happen in one of the powers $w_i^n$, namely if $i \leq 2k$ then these words are not further decomposed, and in such a position we would have $u_t =  u_{t+1}$, or if $w_i$ is further decomposed, then the first item implies that actually $u_t = u_{t+1} = w_{2k}$. The only possibility is then that $u_t u_{t+1}$ appears in the boundary of a $w_{2n-2}$-subword and some word $w_i$ (either $w_i = w_a = w_0$ is the word where $w_{2n}$ differs from $w_{2n+1}$, or $w_i$ begins or ends an $n$th power of $w_i$). Then the claim again follows from the inductive fact that $w_{2n-2}$ begins and ends with $w_{2k}$.

Now, if $Y$ is infinite, then it is not a union of orbits of points of the form $w_i^\Z$ with $i \leq 2k$, thus it must contain long subwords from words $w_{2n}$, which cannot be written as a large power of such $w_i$. Then there exists a word $v$ of length $2|w_{2k}|$ which cannot be written as such a power by the Fine-Wilf theorem. Consider any word $w \in L(Y)$ of length $2 |w_{2k+2}| |w_{2k}| + 2J$
which contains $v$ in its center. The word $w$ appears in some $w_{2n}$, and we can decompose $w_{2n}$ as above into a concatenation of words $w_i$ with $i \leq 2k$. The subword $v$ (of the subword $w$) cannot end up inside a high power of some $w_i$, $i \leq 2k$ in this decomposition, so one of the words $u_t$ at most $J = 2|w_{2k+2}|$ steps away must be equal to $w_{2k}$. The maximum length of the ``steps'' comes from the length of the words $u_t$, which is at most $|w_2k|$. Thus, by the choice of length of $w$, the word $u_t$ is contained inside $w$. Thus, $w$ contains the word $w_{2k}$.

It follows that the language of $Y$ contains every word $w_{2k}$ for all $k$. As discussed, then it contains $w_i^k$ for all $i < k$, thus all the generating periodic points.

Now we show that $X$ has a nonexpansive CAM quotient.

Let $x, y$ be the two points defined above, which differ only at the origin. Then $E = \Delta_X \cup \{x, y\}^2 \subset X^2$ is a closed shift-invariant equivalence relation when $\Delta_X = \{(z, z) \;|\; z \in X\}$. The quotient system $X/E$ is non-expansive because the images of $x_n = w_{2n}^\Z$ and $y_n = w_{2n+1}^\Z$ stay at a uniformly bounded distance throughout their orbit for large $n$ (since $(\sigma^k(x_n), \sigma^k(y_n))$ is either close to the diagonal or close to a shift of the pair $(x, y)$), but they are nevertheless distinct for all $n$.

The $\Z$-action on the quotient system is faithful since $x_n$ still has period exactly $|w_{2n}|$ in the quotient, and every faithful quotient of a CAM system is CAM.
\end{proof}

\section{Residually finite groups}

\begin{theorem}
On every finitely-generated residually finite group $G$, there is a CAM subshift which admits a non-expansive CAM quotient.
\end{theorem}

\begin{proof}
Let us modify the construction to work on general finitely-generated residually finite groups. We will first make a subshift, again with alphabet $A = \{0,1\}$. Our shift convention is $gx_h = x_{g^{-1}h}$. 
We identify $G$ with its right Cayley graph, with nodes $G$ and an edge between $g$ and $gs$ for each generator $s$ in some symmetric generating set, and we consider on $G$ the word metric $d$ from this generating set, i.e.\ path metric in the Cayley graph. Note that this metric is proper, meaning all balls a finite.

Next, use residual finiteness and countability to obtain a decreasing sequence of finite index subgroups $G_i \leq G$ with $\bigcap_i G_i = \{e_G\}$. Pick right coset representatives $R_i$ for $G_i$ inside $G_{i-1}$ so $G_{i-1} = G_i R_i$ (let $G_0 = G$ and $R_0 = \{e_G\}$ for notational convenience). Then $G_i$ has right coset representative $T_i = R_i R_{i-1} \ldots R_1$ inside $G$. Note that by picking $G_{n+1}$ disjoint from a very large $e_G$-centered ball, and picking $R_{i+1}$ in a greedy fashion (always picking the representative closest to the identity until we have a full set of coset representatives), we may assume that $T_n$ contains arbitrarily many arbitrarily large balls at arbitrarily large distances from each other, and at arbitrarily large distances from its boundary. We will make this more quantitative in the analysis below.

Now, we inductively pick a set of points with stabilizer $G_i$. Note that such a point can be identified with a pattern $P$ on $T_i$. Specifically for a pattern $P \in A^{T_i}$, we define $x = x_P \in A^G$ by the formula $x|g T_i = g P$ for all $g \in G_i$ (patterns are shifted in the obvious way: if the domain of $P$ is $D$, that of $gP$ is $gD$, and $gP_{gh} = P_h$). Then $x$ is $G_i$-periodic. Namely, for $g, g' \in G_i, t \in T_i$ we have
\[ gx_{g't} = x_{g^{-1}g't} = g^{-1}g' P_{g^{-1}g' t} = g' P_{g' t} = x_{g' t}. \]

Our subshift will be generated by points $x_P$, where $P \in \mathcal{P}$, where $\mathcal{P}$ contains exactly two patterns $P_{i, 0}, P_{i, 1} \in A^{T_i}$ for each $i = 0, 1, \ldots$. We start with $P_{0,0} = 0, P_{0,1} = 1$ seen as patterns in $A^{R_0}$. (This effectively puts the all-$0$ point and all-$1$ point into the subshift.) The points $x_{P_{k, 0}}, x_{P_{k, 1}}$ will be the analogs of the points $w_{2k}^\Z, w_{2k+1}^\Z$ in the construction we did on $\Z$.

On the level $i+1$, as discussed above we may pick the $G_{i+1}$ to be very sparse compared to $G_i$, so as to have many disjoint large balls inside $T_{i+1}$. In particular, for each pattern $P_{j,a} \in \mathcal{P}$ with $j \leq i$, we may assume we have a separate large ball which we can fill with translates of this pattern, and which is far from the boundary of $T_{i+1}$, and also far from the identity element.

Now, for each previous pattern $P_{j, a}$ we take a large ball $B = B_{j, a} \subset T_{i+1}$ (so that $B_{j, a}$ is far from any other $B_{j', a'}$), and for each $g$ such that $gT_i \subset B$ we recursively split $gT_i$ into shapes $g h R_j$ where $h \in R_i R_{i-1} \cdots R_{j+1}$, and set $P_{i+1, a}|ghR_j = ghP_{j, a}$.
Next, we set $P_{i+1,a}|R_i = P_{i,a}$ (this is the only difference between $P_{i+1,0}$ and $P_{i+1,1}$). Finally, the remaining $gT_i \subset T_{i+1}$ are filled with translates of $P_{i, 0}$. This is illustrated in Figure~\ref{fig:Illustration} for $G = \Z^2$ and $T_i$ large squares.

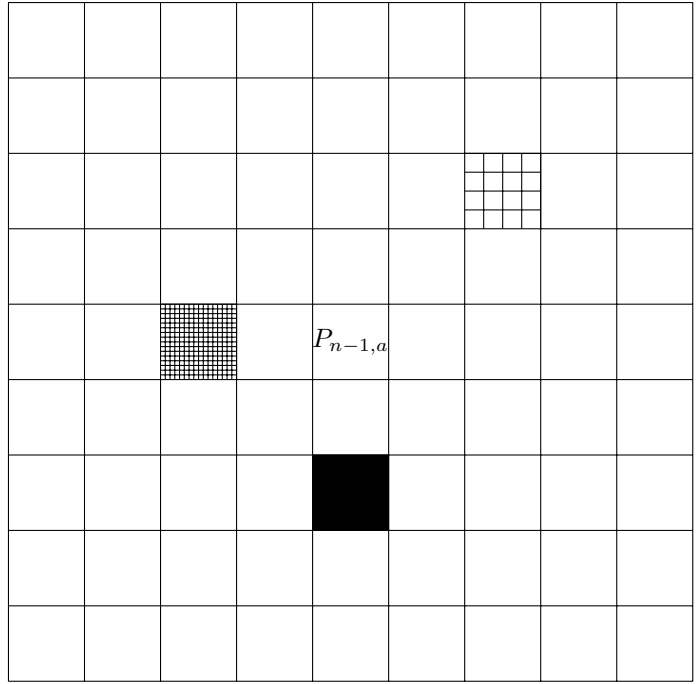
\begin{figure}
\centering
\begin{tikzpicture}
\draw (-1,-1) grid (8,8);
\draw[step = 0.25] (5,5) grid (6,6);
\draw[step = 0.0625] (1,3) grid (2,4);
\draw[fill] (3,1) rectangle (4,2);
\node () at (3.5,3.5) {$P_{n-1, a}$};
\end{tikzpicture}
\caption{An illustration of the construction of $P_{n,a}$. The large white rectangles are filled with $P_{n-1,0}$, and the denser subgrids are filled periodicially with $P_{j',a'}$ for all the previous $j', a'$. At the center we put $P_{n-1,a}$.}
\end{figure}

The argument that the subshift $X$ generated by the points $x_P$ with $P \in \mathcal{P}$ is CAM is exactly analogous to the $\Z$-case. The non-trivial point is that every infinite subsystem contains all of $X$. Thus, suppose $Y \subset X$ is an infinite subshift, and consider some pattern $P_{k, 0}$. The subshift $Y$ is not contained in the union of the orbits of the $x_P$ with $P \in \mathcal{P}' = \{P_{j, a} \;|\; j < k, a\in \{0,1\}\}$. Namely, such a union is finite.

Since every finite subshift on a finitely-generated group is of finite type, every point that is not in an orbit of the $x_P$ with $P \in \mathcal{P}'$ contains a (translate of a) pattern of some bounded radius $r$ (of which there are finitely many, since the metric is proper) from some finite set of forbidden patterns $\mathcal{Q}$. We show that as long as the separations used in the construction are sufficient starting from level $k$, all the patterns $P_{n, a}$ will contain translates of $P_{k,0}$ with bounded graps. Then if the separations are sufficient on all levels, this is true for all $k$ independently.

For the specific $k$, we claim that, analogously to the $\Z$-case, indeed with sufficiently separated choices of $G_{i+1}$ and positioning of the periodic areas above, the following holds true by induction: there exists $R > 0$ such that for all $n$,
\begin{itemize}
\item whenever $g \in T_n$ is at distance at most $2r$ from an element of $G \setminus T_n$, there is an appearance of $P_{k, 0}$ at distance at most $R$ from $g$ in the pattern $P_{n, a}$, and
\item whenever one of the patterns $\mathcal{Q}$ appears at some $g \in T_n$ (inside $P_{n,a}$), there is an appearance of $P_{k, 0}$ at distance at most $R + r$ from $g$.
\end{itemize}

We can simply take $R$ the diameter of $P_{k+1, 0}$ to guarantee this for $n = k+1$. The first item stays true trivially, since when constructing $P_{n+1,a}$ from previous patterns, outside a small area in the center of $T_{n+1}$ where we put the balls $B$ containing other patterns $P_{i, a'}$, we simply repeat the pattern $P_{n, 0}$ where the property holds, so in particular near the boundary we only see the boundary of $P_{n,0}$-patterns. As for the second item, the forbidden patterns will always appear either properly inside one of the $P_{n-1, 0}$-patterns and the property follows from the second item by induction; or they intersect such a pattern properly, and the property again follows from the first item; or they are properly inside a repetition of $P_{j, a'}$ for lower $j$. In this last case, we cannot have have $j < k$ by the choice of $Q$, and thus again either the pattern appears properly inside $P_{j, a'}$ and the claim follows from the second item by induction; or it appears near the boundary and follows from the first item by induction.

A nonexpansive quotient is obtained exactly analogously to the $\Z$-case, by abstractly quotienting out the pair of points differing in exactly one position; and such a pair is in turn obtained as a limit point of the patterns $P_{n,0}, P_{n,1}$.
\end{proof}

\bibliographystyle{plain}
\bibliography{../../../bib/bib}{}

\end{document}